\documentclass{amsart}
\usepackage{amsmath}
\usepackage{amscd}
\usepackage{amssymb}
\usepackage{amsthm}
\RequirePackage{filecontents}
\usepackage{fullpage}
\usepackage[numbers]{natbib}  
\usepackage[draft]{hyperref}

\theoremstyle{plain}
\newtheorem{prop}{Proposition}
\newtheorem{lem}[prop]{Lemma}

\theoremstyle{definition}

\newtheorem{rem}[prop]{Remark}

\author{Brandon Williams }

\address{Department of Mathematics \\ University of California \\ Berkeley, CA 94720}
\email{btw@math.berkeley.edu}

\begin{document}

\nocite{*}

\begin{abstract} For any number $m \equiv 0,1 \, (4)$ we correct the generating function of Hurwitz class number sums $\sum_r H(4n - mr^2)$ to a modular form (or quasimodular form if $m$ is a square) of weight two for the Weil representation attached to a binary quadratic form of discriminant $m$ and determine its behavior in the Petersson scalar product. This modular form arises through holomorphic projection of the zero-value of a nonholomorphic Jacobi Eisenstein series of index $1/m$. When $m$ is prime, we recover the classical Hirzebruch-Zagier series whose coefficients are intersection numbers of curves on a Hilbert modular surface. Finally we calculate certain sums over class numbers by comparing coefficients with an Eisenstein series. \end{abstract}

\title{Vector-valued Hirzebruch-Zagier series and class number sums}

\maketitle

\section{Introduction}

The \textbf{Hurwitz class numbers} $H(n)$ are essentially the class numbers of imaginary quadratic fields. To be more specific, if $-D$ is a fundamental discriminant then $$H(D) = \frac{2 h(D)}{w(D)}$$ where $h(D)$ is the class number of $\mathbb{Q}(\sqrt{-D})$ and $w(D)$ is the number of units in its ring of integers (in particular, $w(D) = 2$ for $D \ne 3,4$). More generally, $$H(n) = \frac{2h(D)}{w(D)} \sum_{d | f} \mu(d) \left( \frac{D}{d} \right) \sigma_1(f/d)$$ if $-n = Df^2$, where $D$ is the discriminant of $\mathbb{Q}(\sqrt{-n})$, and $\mu$ is the M\"obius function, $\sigma_1$ is the divisor sum, and $\left( \frac{\;}{\;} \right)$ is the Kronecker symbol; and by convention one sets $H(0) = -\frac{1}{12}$ and $H(n) = 0$ whenever $n \equiv 1,2 \, (\bmod \, 4)$. Hurwitz class numbers have natural interpretations in terms of equivalence classes of binary quadratic forms or orders in imaginary quadratic fields. We refer to section 5.3 of \cite{C} for more details. \\

Many identities are known to hold between Hurwitz class numbers, the prototypical identity being the Kronecker-Hurwitz relation $$\sum_{r \in \mathbb{Z}} H(4n - r^2) = \sum_{d | n} \max(d, n/d),$$ where we set $H(n) = 0$ if $n < 0$. These and other identities have interpretations in the theory of modular forms. The most influential result in this area is probably Hirzebruch and Zagier's discovery \cite{HZ} that for any prime $p \equiv 1 \, (\text{mod} \, 4)$ the sums $$H_p(n) = \sum_{4n - r^2 \equiv 0 \, (p)} H \left( \frac{4n - r^2}{p} \right)$$ can be corrected to the coefficients of a modular form of weight $2$ and level $\Gamma_0(p)$ and Nebentypus $\chi(n) = \left( \frac{p}{n} \right)$, and that these corrected coefficients can be interpreted as intersection numbers of curves on Hilbert modular surfaces. (The construction of the modular form there also goes through when $p$ is replaced by the discriminant of a real-quadratic number field.) The paper \cite{HZ} is a pioneering use of what are now called mock modular forms. Related techniques have turned out to be effective at deriving other identities among class numbers (among many other things); the papers \cite{BK2}, \cite{M2}, \cite{M1} are some examples of this. \\

As observed by Bruinier and Bundschuh \cite{BB}, there are isomorphisms between the spaces of vector-valued modular forms that transform with the Weil representation attached to a lattice of prime discriminant $p$ and a plus- or minus-subspace (depending on the signature of the lattice) of scalar modular forms of level $\Gamma_0(p)$ and Nebentypus. We prove here that up to a constant factor, the Hirzebruch-Zagier series of level $p$ mentioned above corresponds to a Poincar\'e square series of index $1/p$ (in the sense of \cite{W2}; see also section 2) by computing the latter series directly. One feature of this construction is that $p$ being prime or even a fundamental discriminant is irrelevant: the construction holds and produces a modular form attached to a quadratic form of discriminant $m$ for arbitrary $m \equiv 0,1 \, (\text{mod} \, 4)$ whose coefficients are corrections of the class number sums $\sum_r H(4n - mr^2)$. (However, if $m$ is a perfect square then it will produce a quasimodular form similar to the classical Eisenstein series of weight $2$, rather than a true modular form.) It seems natural to call these vector-valued functions \textbf{Hirzebruch-Zagier series} as well. \\

Our construction starts with a nonholomorphic vector-valued Jacobi Eisenstein series $E_{2,1/m,\beta}^*(\tau,z,s;Q)$ of weight $2$ and index $1/m$ whose Fourier coefficients involve the expressions $H(4n - mr^2)$. (Jacobi forms of fractional index are acceptable when the Heisenberg group also acts through a nontrivial representation.) The action through the Petersson scalar product of the value of the Jacobi Eisenstein series at $z=0$ is straightforward to describe using the usual unfolding argument (e.g. \cite{B}, section 1.2.2) for large enough $\mathrm{Re}(s)$, and it follows for all $s$ by analytic continuation. We construct the Hirzebruch-Zagier series by projecting the zero value $E_{2,1/m,\beta}^*(\tau,0,0;Q)$ orthogonally into the space of cusp forms and then adding the Eisenstein series; neither of these processes change the value of its Petersson scalar product with any cusp form, so this construction makes the behavior of the Hirzebruch-Zagier series with respect to the Petersson scalar product clear for arbitrary $m$. (In the case $m$ is prime, this was left as a conjecture at the end of \cite{HZ}). This method of constructing holomorphic modular forms from real-analytic forms is \textbf{holomorphic projection} and it remains valid for vector-valued modular forms (see also \cite{ORR}). \\

For small values of $m$, there are several examples where the Hirzebruch-Zagier series equals Bruinier's Eisenstein series of weight $2$. By comparing coefficients that are chosen to make the correction term in the Hirzebruch-Zagier series vanish, one can find several identities relating $\sum_r H(4n -mr^2)$ to a twisted divisor sum. A typical example is $$\sum_{r \in \mathbb{Z}} H(4n - 3r^2) = \frac{5}{6} \sigma_1(n,\chi_{12}), \; \; n \equiv 7 \, (\text{mod} \, 12),$$ where $\chi_m(n) = \left( \frac{m}{n} \right)$ is the Kronecker symbol and where $$\sigma_1(n,\chi_m) = \sum_{d | n} d \chi_m(n/d).$$ Additionally, by taking $m = d^2 \in \{4,9,25,49\}$ we give another derivation for identities involving sums of the form $\sum_{r \equiv a \, (d)} H(4n - r^2)$ which were considered in \cite{BK2}, \cite{six}. \\

\noindent \textbf{Acknowledgments:} I would like to thank Kathrin Bringmann and Ken Ono for their feedback on earlier drafts of this note.

\section{Background}

Let $Q$ be an integral nondegenerate quadratic form in $e$ variables with bilinear form $\langle x,y \rangle = \frac{Q(x+y) - Q(x) - Q(y)}{2}$. We associate to $Q$ a \textbf{discriminant form} $(A,Q)$ as follows: letting $S$ be the Gram matrix of $Q$ (that is, $S$ is symmetric with even diagonal and $Q(x) = \frac{1}{2} x^T S x$), we define $A$ as the finite group $A = (S^{-1} \mathbb{Z}^e) / \mathbb{Z}^e$. Then $Q$ induces a well-defined map $$Q : A \rightarrow \mathbb{Q}/\mathbb{Z}, \; Q(x) = \frac{1}{2}x^T S x \, \bmod \, \mathbb{Z}.$$

Attached to $(A,Q)$ is the \textbf{Weil representation} $\rho : Mp_2(\mathbb{Z}) \rightarrow \mathrm{Aut} \, \mathbb{C}[A]$ on the group ring $\mathbb{C}[A]$; letting $\mathfrak{e}_{\gamma},$ $\gamma \in A$ denote the natural basis, this is defined on the generators $S = ( \begin{pmatrix} 0 & -1 \\ 1 &0 \end{pmatrix}, \sqrt{\tau} )$ and $T = ( \begin{pmatrix} 1 & 1 \\ 0 & 1 \end{pmatrix}, 1)$ by $$\rho(S) \mathfrak{e}_{\gamma} = \frac{1}{\sqrt{|A|}} \mathbf{e}\Big( \mathrm{sig}(Q) / 8 \Big) \sum_{\beta \in A} \mathbf{e}\Big( \langle \gamma, \beta \rangle \Big) \mathfrak{e}_{\beta}$$ and $$\rho(T) \mathfrak{e}_{\gamma} = \mathbf{e}\Big( -Q(\gamma) \Big) \mathfrak{e}_{\gamma}.$$ (We are using the convention of \cite{Sch}; elsewhere the Weil representation may denote the dual of $\rho$.) Here $\mathrm{sig}(Q)$ denotes the signature of $Q$ and $\mathbf{e}(x) = e^{2\pi i x}$. This representation is unitary with respect to the scalar product on $\mathbb{C}[A]$ that makes $\mathfrak{e}_{\gamma}$, $\gamma \in A$, an orthonormal basis. Note that isomorphic discriminant forms produce the same Weil representation. If $\mathrm{sig}(Q)$ is even, then $\rho$ comes from a representation of $SL_2(\mathbb{Z})$; further, if $\mathrm{sig}(Q)$ is a multiple of $4$ then $\rho$ comes from a representation of $PSL_2(\mathbb{Z})$. \\

By $M_k(Q)$, $k \in \frac{1}{2}\mathbb{Z}$, we denote the space of \textbf{modular forms of weight $k$ for the Weil representation attached to $Q$}, which are the holomorphic functions $f : \mathbb{H} \rightarrow \mathbb{C}[A]$ that satisfy $$f\Big( \frac{a \tau + b}{c \tau + d} \Big) = (c \tau + d)^k \rho(M) f(\tau), \; \; M = (\begin{pmatrix} a & b \\ c & d \end{pmatrix}, \sqrt{c \tau + d}) \in Mp_2(\mathbb{Z}),$$ where $(c \tau + d)^k = (\sqrt{c \tau + d})^{2k}$ if $k$ is half-integral, together with a growth condition at $\infty$ that is equivalent to $f$ having a Fourier expansion of the form $$f(\tau) = \sum_{\gamma \in A} \sum_{\substack{n \in \mathbb{Z} - Q(\gamma) \\ n \ge 0}} c(n,\gamma) q^n \mathfrak{e}_{\gamma}, \; \; c(n,\gamma) \in \mathbb{C}.$$ The first condition can be abbreviated by $f|_{k,\rho} M = f$ for all $M \in Mp_2(\mathbb{Z})$, where $|_{k,\rho}$ is the Petersson slash operator $$f |_{k,\rho} M (\tau) = (c \tau + d)^{-k} \rho(M)^{-1} f(\tau).$$ We call $f$ a \textbf{cusp form} if in addition all constant terms $c(0,\gamma)$ are zero. \\

\begin{rem} Following proposition 4.5 of \cite{Sch}, if $N$ is the level of the discriminant form $(A,Q)$ then in the case of even signature the action of $\Gamma_0(N)$ through the Weil representation is $$\rho(M) \mathfrak{e}_{\gamma} = \chi(M) \mathbf{e}\Big( -bd Q(\gamma) \Big) \mathfrak{e}_{d \gamma}, \; \; M = \begin{pmatrix} a & b \\ c & d \end{pmatrix} \in SL_2(\mathbb{Z}), \; \; c \equiv 0 \, (N).$$ Here $\chi$ is a quadratic character naturally associated to the quadratic form $Q$. It follows from this that if $f(\tau) = \sum_{\gamma \in A} f_{\gamma}(\tau) \mathfrak{e}_{\gamma}$ is a modular form for the Weil representation attached to $Q$ then $$\sum_{\gamma \in A} f_{\gamma}(N \tau)$$ is a (scalar-valued) modular form of level $\Gamma_0(N)$ and Nebentypus $\chi$. In the case $\mathrm{sig}(Q) \equiv 0 \, (\text{mod} \, 4)$ and $\mathrm{discr}(Q) = p$ is prime (and therefore the level of $Q$ is also $p$), it was shown in \cite{BB} that the character is $\chi(n) = \left( \frac{p}{n} \right)$ and that this correspondence is an isomorphism to the ``plus space" of modular forms with Nebentypus in which the coefficient of $q^n$ is zero if $\chi(n) = -1$.
\end{rem}

For $k \ge 5/2$, the \textbf{Eisenstein series} of weight $k$ attached to $Q$ is defined by $$E_k(\tau;Q) = \sum_{M \in \Gamma_{\infty} \backslash Mp_2(\mathbb{Z})} \mathfrak{e}_0 \Big|_{k,\rho} M,$$ where $\mathfrak{e}_0$ is interpreted as a constant function and $\Gamma_{\infty}$ is the subgroup of $Mp_2(\mathbb{Z})$ generated by $T$ and $Z = ( \begin{pmatrix} -1 & 0 \\ 0 & -1 \end{pmatrix}, i).$ Note that this construction is identically zero unless $2k + \mathrm{sig}(Q) \equiv 0 \, (\text{mod} \, 4).$ When $k \le 2$, this no longer converges (absolutely); however, one can instead consider the nonholomorphic Eisenstein series $$E_k^*(\tau,s;Q) = \sum_{M \in \Gamma_{\infty} \backslash Mp_2(\mathbb{Z})} (y^s \mathfrak{e}_0) \Big|_{k,\rho} M,$$ which has an analytic continuation in $s$ beyond the region of convergence. These series were considered in \cite{W3}; in weight $k = 3/2$, setting $s=0$ results in a harmonic Maass form $E_{3/2}^*(\tau,0;Q)$ whose Fourier expansion takes the form $$E_{3/2}^*(\tau,0;Q) = \sum_{\gamma \in A} \sum_{\substack{n \in \mathbb{Z} - Q(\gamma) \\ n \ge 0}} c(n,\gamma) q^n \mathfrak{e}_{\gamma} + \frac{1}{\sqrt{y}} \sum_{\substack{n \in \mathbb{Z} - Q(\gamma) \\ n \le 0}} a(n,\gamma) \beta(-4\pi n y) q^n \mathfrak{e}_{\gamma},$$ where $c(n,\gamma), a(n,\gamma) \in \mathbb{C}$ and $\beta$ is a special function (essentially an incomplete Gamma function) $$\beta(x) = \frac{1}{16\pi} \int_1^{\infty} u^{-3/2} e^{-xu} \, \mathrm{d}u.$$ When $k=2$, the zero-value $E_2^*(\tau,0;Q)$ is a quasimodular form whose Fourier expansion takes the form $$E_2^*(\tau,0;Q) = \sum_{\gamma \in A} \sum_{\substack{n \in \mathbb{Z} - Q(\gamma) \\ n \ge 0}} c(n,\gamma) q^n \mathfrak{e}_{\gamma} + \frac{1}{y} \cdot \mathrm{const}.$$ (The classical weight $2$ Eisenstein series $E_2^*(\tau) = 1 - 24 \sum_{n=1}^{\infty} \sigma_1(n) q^n - \frac{3}{\pi y}$ is a special case of this.) \\

The most important example of a weight $3/2$ Eisenstein series is the Zagier Eisenstein series, which we consider in the form $$E_{3/2}^*(\tau,0;x^2) = -12 \sum_{\gamma \in \frac{1}{2} \mathbb{Z}/\mathbb{Z}} \sum_{n \in \mathbb{Z} - \gamma^2} H(4n) q^n \mathfrak{e}_{\gamma} + \frac{1}{\sqrt{y}} \sum_{\gamma \in \frac{1}{2} \mathbb{Z}/\mathbb{Z}} \sum_{n \in \mathbb{Z} - \gamma^2} a(n,\gamma) \beta(-4\pi n y) q^n \mathfrak{e}_{\gamma}.$$ Here $H(n)$ is the \textbf{Hurwitz class number} from the introduction, and $a(0,0) = -24$; $a(n,\gamma) = -48$ if $-n$ is a rational square and $a(n,\gamma) = 0$ otherwise. \\

We will also need to consider Jacobi forms for the ``twisted Weil representations'' described in \cite{W1}. The Heisenberg group is the group $\mathcal{H}$ with underlying set $\mathbb{Z}^3$ and group action $$(\lambda_1,\mu_1,t_1) \cdot (\lambda_2, \mu_2, t_2) = (\lambda_1+\lambda_2, \mu_1 + \mu_2, t_1 + t_2 + \lambda_1 \mu_2 - \lambda_2 \mu_1).$$ The metaplectic group $Mp_2(\mathbb{Z})$ acts on $\mathcal{H}$ from the right by $$(\lambda, \mu,t) \cdot \begin{pmatrix} a & b \\ c & d \end{pmatrix} = (a \lambda + c \mu, b \lambda + d\mu, t),$$ and the meta-Jacobi group is defined as the semidirect product $\mathcal{J} = \mathcal{H} \rtimes Mp_2(\mathbb{Z})$ by this action. For any $\beta \in A$, there is a unitary \textbf{Schr\"odinger representation} $\sigma_{\beta}$ of $\mathcal{H}$ on $\mathbb{C}[A]$ given by $$\sigma_{\beta}(\lambda, \mu,t) \mathfrak{e}_{\gamma} = \mathbf{e}\Big(-\mu \langle \beta, \gamma \rangle + (\lambda \mu - t) Q(\beta) \Big) \mathfrak{e}_{\gamma - \lambda \beta}.$$ (This is the dual of the representation discussed in \cite{W1}.) This representation is compatible with the Weil representation and one obtains a representation $$\rho_{\beta} : \mathcal{J} \rightarrow \mathrm{Aut} \, \mathbb{C}[A]$$ that restricts to $\sigma_{\beta}$ on $\mathcal{H}$ and to $\rho$ on $Mp_2(\mathbb{Z})$. \\

A \textbf{Jacobi form} of weight $k$ and index $m \in \mathbb{Z} - Q(\beta)$ for the representation $\rho_{\beta}$ is a holomorphic function of two variables $\Phi : \mathbb{H} \times \mathbb{C} \rightarrow \mathbb{C}[A]$ with the following properties: \\ (i) For any $M = (\begin{pmatrix} a & b \\ c & d \end{pmatrix}, \sqrt{c \tau + d} ) \in Mp_2(\mathbb{Z})$, $$\Phi \Big( \frac{a \tau + b}{c \tau +d}, \frac{z}{c \tau + d} \Big) = (c \tau + d)^k \mathbf{e}\Big( \frac{mcz^2}{c \tau + d} \Big) \rho(M) \Phi(\tau,z);$$ (ii) For any $(\lambda,\mu,t) \in \mathcal{H}$, $$\Phi(\tau,z+\lambda \tau + \mu) = \mathbf{e}\Big( - m\lambda^2 \tau - 2m \lambda z - m(\lambda \mu + t) \Big) \sigma_{\beta}(\lambda,\mu,t) \Phi(\tau,z);$$ (iii) The Fourier series of $\Phi$, $$\Phi(\tau,z) = \sum_{\gamma \in A} \sum_{n \in \mathbb{Z} - Q(\gamma)} \sum_{r \in \mathbb{Z} - \langle \gamma, \beta \rangle} c(n,r,\gamma) q^n \zeta^r \mathfrak{e}_{\gamma}, \; \; q = \mathbf{e}(\tau), \; \zeta = \mathbf{e}(z)$$ has $c(n,r,\gamma) = 0$ whenever $n < r^2 / 4m$. \\

Jacobi forms as above arise naturally through the Fourier-Jacobi expansion of vector-valued Siegel modular forms. If $X \begin{pmatrix} \tau_1 & z \\ z & \tau_2 \end{pmatrix}$ is a Siegel modular form of weight $k$ that transforms with respect to the Weil representation of $Mp(4;\mathbb{Z})$ on $\mathbb{C}[A] \otimes \mathbb{C}[A]$ (which encodes the transformation of the genus two Siegel theta function of $-Q$, if $Q$ is negative definite), one can write out the Fourier series of $X$ with respect to the variable $\tau_2$ as $$X \begin{pmatrix} \tau_1 & z \\ z & \tau_2 \end{pmatrix} = \sum_{\beta \in A} \sum_{m \in \mathbb{Z} - Q(\beta)} \Phi_{m,\beta}(\tau_1,z) \otimes q_2^m \mathfrak{e}_{\beta}, \; \; q_2 = \mathbf{e}(\tau_2),$$ and $\Phi_{m,\beta}$ is a Jacobi form of weight $k$ and index $m$ for $\rho_{\beta}$. \\

One particularly important Jacobi form is the Jacobi Eisenstein series $$E_{k,m,\beta}(\tau,z;Q) = \sum_{(M,\zeta) \in \mathcal{J}_{\infty} \backslash \mathcal{J}} \mathfrak{e}_0 \Big|_{k,m,\rho_{\beta}} (M,\zeta), \; \; k \ge 3;$$ here, $|_{k,m,\rho_{\beta}}$ is a Petersson slash operator defined similarly to $|_{k,\rho}$, and $\mathcal{J}_{\infty}$ is the stabilizer of $\mathfrak{e}_0$ in $\mathcal{J}$. One can extend this to smaller weights by defining $E_{k,m,\beta}^*(\tau,z,s;Q) = \sum_{(M,\zeta)} (y^s \mathfrak{e}_0) |_{k,m,\rho_{\beta}} (M,\zeta).$ When $k=2$, the value at $s=0$ has a Fourier expansion of the form \begin{align*} E_{2,m,\beta}^*(\tau,z,0;Q) &= \sum_{\gamma \in A} \sum_{n \in \mathbb{Z} - Q(\gamma)} \sum_{r \in \mathbb{Z} - \langle \gamma, \beta \rangle} c(n,r,\gamma) q^n \zeta^r \mathfrak{e}_{\gamma} + \\ &\quad\quad\quad\quad + \frac{1}{\sqrt{y}} \sum_{\gamma \in A} \sum_{n \in \mathbb{Z} - Q(\gamma)} \sum_{r \in \mathbb{Z} - \langle \gamma, \beta \rangle} A(n,r,\gamma) \beta(\pi y (r^2 / m - 4n)) q^n \zeta^r \mathfrak{e}_{\gamma},\end{align*} where $A(n,r,\gamma)$ are constants that are zero whenever $(r^2 - 4mn) |A|$ is not a rational square. \\

For $k \ge 3$, the zero-values $Q_{k,m,\beta}(\tau;Q) = E_{k,m,\beta}(\tau,0;Q)$ are easily seen to be modular forms for $\rho$. As observed in \cite{W1} they have an interesting characterization in terms of the Petersson scalar product $$(f,g) = \int_{SL_2(\mathbb{Z}) \backslash \mathbb{H}} \langle f(\tau), g(\tau) \rangle y^{k-2} \, \mathrm{d}x \, \mathrm{d}y, \; \; \tau = x + iy;$$ namely, $Q_{k,m,\beta}(\tau;Q) - E_k(\tau;Q)$ is a cusp form and $$(f,Q_{k,m,\beta}) = \frac{2 \cdot \Gamma(k-1)}{(4m \pi)^{k-1}} \sum_{\lambda=1}^{\infty} \frac{c(\lambda^2 m, \lambda \beta)}{\lambda^{2k-2}}$$ for all cusp forms $f(\tau) = \sum_{\gamma,n} c(n,\gamma) q^n \mathfrak{e}_{\gamma}.$ This characterization was essentially taken as a definition for $Q_{k,m,\beta}$ in weights $k=3/2,2,5/2$ in \cite{W1}, \cite{W3}; in particular, in weight $k=2$ one can obtain $Q_{2,m,\beta}$ through holomorphic projection of the zero-value $E_{2,m,\beta}^*(\tau,0,0;Q)$. \\

\section{The case $m \equiv 1 \, \bmod \, 4$}

Fix any number $m \equiv 1 \, \bmod \, 4$ and consider the quadratic form $Q(x,y) =  x^2 + xy - \frac{m-1}{4} y^2$ of discriminant $m.$ There is a unique pair of elements $\pm \beta \in A$ in the associated discriminant form with $Q(\beta) = 1 - \frac{1}{m}$; they are represented by $\pm (-1/m, 2/m)$. (In particular, the discriminant form is cyclic and these elements are generators, and $Q$ takes values in $\frac{1}{m} \mathbb{Z} / \mathbb{Z}$.) We will also consider the ternary quadratic form $$\mathbf{Q}(x,y,z) = Q(x,y) + 2xz + z^2,$$ which has discriminant $$\mathrm{discr}(\mathbf{Q}) = \mathrm{det} \begin{pmatrix} 2 & 1 & 2 \\ 1 & -\frac{m-1}{2} & 0 \\ 2 & 0 & 2 \end{pmatrix} = -2.$$

Comparing the coefficient formulas for the Jacobi Eisenstein series of index $1/m$ (\cite{W2}, section 3) and the usual Eisenstein series \cite{BK} (which in weight $3/2$ was considered in \cite{W3}), we see that the coefficient of $q^n \zeta^r$ in $E_{2,1/m,\beta}^*(\tau,z,s;Q)$ equals the coefficient of $q^{n - mr^2 / 4}$ in $E_{3/2}^*(\tau,s;\mathbf{Q})$. To be more precise, we should consider the coefficients of $q^n \zeta^r \mathfrak{e}_{\gamma}$ for elements $\gamma \in A$ instead; however, the condition $r \in \mathbb{Z} - \langle \gamma, \beta \rangle$ determines $\langle \gamma, \beta \rangle \in \mathbb{Q}/\mathbb{Z}$, and due to our choice of $\beta$ this determines $\gamma$ uniquely. Both coefficient formulas involve zero-counts of quadratic polynomials modulo prime powers and $\mathbf{Q}$ is chosen to make these zero-counts equal; specifically, for all $\gamma \in A$ and $n \in \mathbb{Z} - Q(\gamma)$, $r \in \mathbb{Z} - \langle \gamma, \beta \rangle$, \begin{align*} &\;\,\quad \# \{(v,\lambda) \in \mathbb{Z}^3 \, \text{mod} \, p^k: \; Q(v + \lambda \beta - \gamma) + \lambda^2 / m - r \lambda + n \equiv 0 \} \\ &= \#\{v \in \mathbb{Z}^3 \, \text{mod} \, p^k: \; \mathbf{Q}(v - \gamma_r) + (n - mr^2 / 4) \equiv 0 \}, \; \; \gamma_r = \Big(\gamma - \frac{rm}{2} \beta, \frac{rm}{2}\Big) \in (\mathbb{Q}/\mathbb{Z})^3, \end{align*} where $\gamma_r$ is in the discriminant group of $\mathbf{Q}$ and $n - \frac{mr^2}{4} \in \mathbb{Z} - \mathbf{Q}(\gamma_r)$. On the level of matrices, letting $S = \begin{pmatrix} 2 & 1 \\ 1 & -\frac{m-1}{2} \end{pmatrix}$ be the Gram matrix of $Q$, this follows because the Gram matrix of $\mathbf{Q}$ has block form $$\mathrm{Gram}(\mathbf{Q}) = \begin{pmatrix} 2& 1 & 2 \\ 1 & -\frac{m-1}{2} & 0 \\ 2 & 0 & 2 \end{pmatrix} =  \begin{pmatrix} S & S \beta \\ \beta^T S & 2(\frac{1}{m} + \beta^T S \beta) \end{pmatrix}$$ for the representative $\beta = (\frac{m-1}{m}, \frac{2}{m}) \in \mathbb{Q}^2$. (See also Remark 17 of \cite{W1}.) \\

 Since $|\mathrm{discr}(\mathbf{Q})| = 2$ and $\mathrm{sig}(\mathbf{Q}) = 1$, the discriminant form of $\mathbf{Q}$ is isomorphic to that of $x^2$ and so the nonholomorphic weight $3/2$ Eisenstein series attached to it is the Zagier Eisenstein series in which the coefficient of $q^n$ is $H(4n)$. Evaluating at $s=0$ and using the previous paragraph, we find \begin{equation} \begin{split} E_{2,1/m,\beta}^*(\tau,z,0;Q) &= -12 \sum_{\gamma \in A} \sum_{n \in \mathbb{Z} - Q(\gamma)} \sum_{r \in \mathbb{Z} - \langle \gamma, \beta \rangle} H(4n - mr^2) q^n \zeta^r \mathfrak{e}_{\gamma} + \\ &\quad\quad\quad\quad + \frac{1}{\sqrt{y}} \sum_{\gamma \in A} \sum_{n \in \mathbb{Z} - Q(\gamma)} \sum_{r \in \mathbb{Z} - \langle \gamma, \beta \rangle} A(n,r,\gamma) \beta(\pi y (mr^2 - 4n))  q^n \zeta^r \mathfrak{e}_{\gamma}, \end{split} \end{equation} where the coefficients $A(n,r,\gamma)$ are given by $$ A(n,r,\gamma) = \begin{cases} -24: & mr^2 = 4n; \\ -48: & mr^2 - 4n \; \text{is a nonzero square}; \\ 0: & \text{otherwise}. \end{cases}$$

The main result of \cite{W2} is a coefficient formula for the Poincar\'e square series $Q_{2,d,\beta}(\tau)$ for any $\beta \in A$ and $d \in \mathbb{Z} - Q(\beta)$. If $$E_{2,d,\beta}^*(\tau,0,0;Q) = \sum_{\gamma \in A} \sum_{n \in \mathbb{Z} - Q(\gamma)} \sum_{r \in \mathbb{Z} - \langle \gamma, \beta \rangle} c(n,r,\gamma) q^n \zeta^r \mathfrak{e}_{\gamma} + \frac{1}{\sqrt{y}} \sum_{\gamma \in A} \sum_{n \in \mathbb{Z} - Q(\gamma)} \sum_{r \in \mathbb{Z} - \langle \gamma, \beta \rangle} A(n,r,\gamma) \beta(\pi y (r^2 / d - 4n)) q^n \zeta^r \mathfrak{e}_{\gamma},$$ then the coefficient $C(n,\gamma)$ of $q^n \mathfrak{e}_{\gamma}$ in $Q_{2,d,\beta}$ is $$C(n,\gamma) = \sum_{r \in \mathbb{Z} - \langle \gamma, \beta \rangle} c(n,r,\gamma) + \frac{1}{8 \sqrt{d}} \sum_{r \in \mathbb{Z} - \langle \gamma, \beta \rangle} \Big[ A(n,r,\gamma) \Big( |r| - \sqrt{r^2 - 4dn} \Big) \Big].$$ In our case where $d = 1/m$, these coefficients are \begin{equation} -12 \sum_{r \in \mathbb{Z} - \langle \gamma, \beta \rangle} H(4n - mr^2) - 6\sqrt{m} \sum_{\substack{r \in \mathbb{Z} - \langle \gamma, \beta \rangle \\ mr^2 - 4n = \square}} \Big( |r| - \sqrt{r^2 - 4n/m} \Big) + \begin{cases} 12 \sqrt{n}: & \exists r \in \mathbb{Z} - \langle \gamma, \beta \rangle \; \text{with} \; mr^2 = 4n; \\ 0: & \text{otherwise}. \end{cases} \end{equation}

Note that $\frac{m}{2} \Big( |r| - \sqrt{r^2 - 4n/m} \Big)$ is always an algebraic integer when $mr^2 - 4n$ is square: if $m$ itself is square then this is clear, and otherwise its conjugate is $\frac{m}{2} \Big( |r| + \sqrt{r^2 - 4n/m} \Big)$, so its trace is $m|r| \in m \mathbb{Z} \pm m \langle \gamma, \beta \rangle \subseteq \mathbb{Z}$ and its norm is $mn \in m \mathbb{Z} - m Q(\gamma) \subseteq \mathbb{Z}.$ Conversely, if $m$ is squarefree then $$\frac{m}{2} \Big( |r| - \sqrt{r^2 - 4n/m} \Big), \; \; r \in \mathbb{Z} \pm \langle \gamma, \beta \rangle, \; \gamma \in A$$ runs exactly through the values taken by $\min(\lambda, \overline{\lambda})$, where $\lambda$ is a positive integer of $\mathcal{O}_K$, $K = \mathbb{Q}(\sqrt{m})$ with positive conjugate $\overline{\lambda}$ that has norm $mn$; however it double-counts $\sqrt{mn}$ if $r = \pm \sqrt{4n/m}$ occur in the sum. This allows one to remove the additional term $12\sqrt{n}$ in the formula (2). In some sense this remains true for $m = 1$ (with trivial discriminant form): then $\frac{1}{2} (|r| - \sqrt{r^2 - 4n})$, $r \in \mathbb{Z}$ takes the values $\min(d,n/d)$ where $d$ runs through divisors of $n$ in $\mathbb{Z}$ but it double-counts $\sqrt{n}$ if $n$ is square.) \\

In this way, we obtain for $m=1$ an identity equivalent to the Kronecker-Hurwitz relations: $$E_2(\tau) = 1 - 24 \sum_{n=1}^{\infty} \sigma_1(n) q^n = Q_{2,1,0}(\tau) = 1 -12 \sum_{n \in \mathbb{Z}} \Big( \sum_{d | n} \min(d,n/d) + \sum_{r \in \mathbb{Z}} H(4n - r^2) \Big) q^n,$$ whereas if $m = p$ is a prime, we obtain a vector-valued form of the Hirzebruch-Zagier series: $$Q_{2,1/p,\beta}(\tau) = 1 - 12 \sum_{\gamma \in A} \varepsilon_{\gamma} \sum_{n \in \mathbb{Z} - Q(\gamma)} \Big( \sum_{r \in \mathbb{Z} - \langle \gamma, \beta \rangle} H(4n - pr^2) + \frac{1}{\sqrt{p}} \sum_{\substack{ \lambda \in \mathcal{O}_K \\ \lambda \gg 0 \\ \lambda \overline{\lambda} = pn}} \min(\lambda, \overline{\lambda}) \Big) q^n \mathfrak{e}_{\gamma},$$ where $K = \mathbb{Q}(\sqrt{p})$; $\overline{\lambda}$ is the conjugate of $\lambda$; $\lambda \gg 0$ means that both $\lambda, \overline{\lambda}$ are positive; and finally we set $\varepsilon_{\gamma} = 1$ if $\gamma = 0$ and $\varepsilon_{\gamma} = 1/2$ otherwise. (The factors $\varepsilon_{\gamma}$ come from the fact that relating the sum over $r$ to a divisor sum requires both congruences $r \in \mathbb{Z} \pm \langle \gamma, \beta \rangle$; but the coefficients $c(n,\gamma)$ of any modular form satisfy $c(n,\gamma) = c(n,-\gamma)$ by our assumption $2k + \mathrm{sig}(Q) \equiv 0 \, (4)$ on their weight.) As mentioned in Remark 1 the identification between modular forms attached to quadratic forms of prime discriminant and the plus space of modular forms with Nebentypus is essentially given by summing together all components and replacing $n$ by $n/p$ in the coefficient formula; it is not difficult to see that the image of $Q_{2,1/p,\beta}$ under this identification is $$-12 \varphi_p = 1 - 12 \sum_{n \in \mathbb{Z}} \Big[ \sum_{\substack{r \in \mathbb{Z} \\ 4n - r^2 \equiv 0 \, (p)}} H\Big( \frac{4n - r^2}{p} \Big) + \frac{1}{\sqrt{p}} \sum_{\substack{\lambda \in \mathcal{O}_K \\ \lambda \gg 0 \\ \lambda \overline{\lambda} = n}} \min(\lambda, \overline{\lambda}) \Big] q^n \in M_2( \Gamma_0(p), \chi),$$ where $\varphi_p$ is the function of Hirzebruch and Zagier's paper \cite{HZ}.

\section{The case $m \equiv 0 \, \bmod \, 4$}

Our procedure in this case is nearly the same, but we consider instead the quadratic form $Q(x,y) = x^2 - \frac{m}{4} y^2$ of discriminant $m$. There is again an element $\beta \in A$ with $Q(\beta) = 1 - \frac{1}{m};$ in this case, one can choose the representative $(0,2/m).$ (Note that this discriminant form is not cyclic. Also, $\beta$ is not necessarily unique; but any other choice of $\beta$ will give a similar result.) We also consider the ternary quadratic form $\mathbf{Q}(x,y,z) = Q(x,y) - yz$ which has discriminant $$\mathrm{discr}(\mathbf{Q}) = \mathrm{det} \begin{pmatrix} 2 & 0 & 0 \\ 0 & -m/2 & -1 \\ 0 & -1 & 0 \end{pmatrix} = -2.$$ Comparing coefficient formulas between $E_2^*(\tau,z,s;Q)$ and $E_{3/2}^*(\tau;\mathbf{Q})$ gives exactly the same formula as equation (1) in the previous section: \begin{align*} E_{2,1/m,\beta}^*(\tau,z,0;Q) &= -12 \sum_{\gamma \in A} \sum_{n \in \mathbb{Z} - Q(\gamma)} \sum_{r \in \mathbb{Z} - \langle \gamma, \beta \rangle} H(4n - mr^2) q^n \zeta^r \mathfrak{e}_{\gamma} + \\ &\quad\quad\quad\quad + \frac{1}{\sqrt{y}} \sum_{\gamma \in A} \sum_{n \in \mathbb{Z} - Q(\gamma)} \sum_{r \in \mathbb{Z} - \langle \gamma, \beta \rangle} A(n,r,\gamma) \beta(\pi y (m r^2 - 4n)) q^n \zeta^r \mathfrak{e}_{\gamma}, \end{align*} and therefore the same coefficient formula from equation (2) produces a modular form for $Q$.

\section{Formulas for class number sums}

In this section we compute values of $m$ where the relevant space of weight $2$ modular forms is one-dimensional and therefore the Hirzebruch-Zagier series $Q_{2,1/m,\beta}$ equals the Eisenstein series. We obtain formulas for class number sums by considering those exponents $n$ for which the corrective term $$-6\sqrt{m} \sum_{mr^2 - 4n = \square} (|r| - \sqrt{r^2 - 4n/m})$$ above vanishes. First we will list the numbers $m$ for which the cusp space $S_2(Q)$ attached the quadratic forms we considered above vanishes.

\begin{lem} (i) Suppose $m \equiv 1$ mod $4$ and let $Q(x,y) = x^2 + xy - \frac{m-1}{4} y^2$. Then the cusp space $S_2(Q)$ vanishes if and only if $m \le 25$. \\ (ii) Suppose $m \equiv 0$ mod $4$ and let $Q(x,y) = x^2 - \frac{m}{4} y^2$. Then the cusp space $S_2(Q)$ vanishes if and only if $m \le 20$.
\end{lem}
\begin{proof} Table 7 of \cite{BEF} lists the genus symbols of all discriminant forms of signature $0$ mod $8$ with at most four generators for which the space of weight two cusp forms vanishes. We only need to find the values of $m$ for which the discriminant form of $Q$ appears in their table.
\end{proof}

In particular, when $m \le 21$ or $m=25$, due to the lack of cusp forms the Hirzebruch-Zagier series $Q_{2,1/m,\beta}(\tau;Q)$ equals the Eisenstein series $E_2(\tau;Q)$ in which the coefficient of $q^n \mathfrak{e}_{\gamma}$ is a multiple of the twisted divisor sum $$\sigma_1(n d_{\gamma}^2,\chi_m) = \sum_{d | n d_{\gamma}^2} d \cdot \left( \frac{m}{nd_{\gamma}^2 / d} \right),$$ where $d_{\gamma}$ is the denominator of $\gamma$ (i.e. the smallest number such that $d_{\gamma} \gamma = 0$ in $A$) and $\chi_m = \left( \frac{m}{-} \right)$ is the quadratic character attached to $\mathbb{Q}(\sqrt{m})$; and these multiples are constant when $n$ is restricted to certain congruence classes. This leads to numerous identities relating class number sums of the form $\sum_{r \in \mathbb{Z} - \langle \gamma, \beta \rangle} H(4n - mr^2)$ (even in some cases where $n$ is not integral!) to twisted divisor sums. \\

The simplest identities arise by comparing the components of $\mathfrak{e}_0$ in both series and restricting to odd integers $n$ for which in addition $\chi_m(n) = -1$ (which is never true for $m = 1,4,9,16,25$); in these cases, the ``correction term" in $Q_{2,1/m,\beta}(\tau;Q)$ vanishes and its coefficient of $q^n \mathfrak{e}_0$ is $-12 \sum_{r \in \mathbb{Z}} H(4n - mr^2).$ This is then a constant multiple of $\sigma_1(n,\chi_m)$ depending on the remainder of $n$ mod $m$. The constant multiple can be computed by studying the formula of \cite{BK} carefully but it is easier to compute by plugging in just one value of $n$. We list the results one can obtain with this argument: \\

\noindent (1) $m = 5$: for $n \equiv 3,7 \, (\text{mod} \, 10)$, $$\sum_{r \in \mathbb{Z}} H(4n - 5r^2) = \frac{5}{3} \sigma_1(n,\chi_5).$$

\noindent (2) $m = 8$: for $n \equiv 3,5 \, (\text{mod} \, 8)$, $$\sum_{r \in \mathbb{Z}} H(4n - 2r^2) = \sum_{r \in \mathbb{Z}} H(4n - 8r^2) = \frac{7}{6} \sigma_1(n,\chi_8).$$

\noindent (3) $m = 12$: for $n \equiv 5 \, (\text{mod} \, 12)$, $$\sum_{r \in \mathbb{Z}} H(4n - 3r^2) = \sum_{r \in \mathbb{Z}} H(4n - 12r^2) = \sigma_1(n,\chi_{12}),$$ and for $n \equiv 7 \, (\text{mod} \, 12)$, $$\sum_{r \in \mathbb{Z}} H(4n - 3r^2) = \sum_{r \in \mathbb{Z}} H(4n - 12r^2) = \frac{5}{6} \sigma_1(n,\chi_{12}).$$

\noindent (4) $m = 13$: for $n \equiv 5,7,11,15,19,21 \, (\text{mod} \, 26)$, $$\sum_{r \in \mathbb{Z}} H(4n - 13r^2) = \sigma_1(n,\chi_{13}).$$

\noindent (5) $m=17$: for $n \equiv 3,5,7,11,23,27,29,31 \, (\text{mod} \, 34)$, $$\sum_{r \in \mathbb{Z}} H(4n - 17r^2) = \frac{2}{3} \sigma_1(n,\chi_{17}).$$

\noindent (6) $m = 20$: for $n \equiv 3,7,13,17 \, (\text{mod} \, 20)$, $$\sum_{r \in \mathbb{Z}} H(4n - 20r^2) = \frac{2}{3} \sigma_1(n,\chi_{20}).$$

\noindent (7) $m = 21$: for $n \equiv 11,23,29 \, (\text{mod} \, 42)$, $$\sum_{r \in \mathbb{Z}} H(4n - 21r^2) = \sigma_1(n,\chi_{21})$$ and for $n \equiv 13,19,31 \, (\text{mod} \, 42)$, $$\sum_{r \in \mathbb{Z}} H(4n - 21r^2) = \frac{2}{3} \sigma_1(n,\chi_{21}).$$

\begin{rem} There are some values of $m$ where the Eisenstein series does not equal the Hirzebruch-Zagier series (one should not expect it to when $\mathrm{dim}\, S_2(Q) > 0$) but where one can still obtain some information by comparing coefficients within arithmetic progressions, yielding more identities than those above. (Note that if $f(\tau) = \sum_{n \in \mathbb{Z}} c(n) q^n$ is a modular form of some level $N$, then restricting to an arithmetic progression produces a modular form $\sum_{n \equiv r \, (\text{mod} \, d)} c(n) q^n$ of the same weight and level $Nd^2$; so one can always check whether the coefficients of two modular forms agree in an arithmetic progression by computing finitely many coefficients.) In particular, we do not claim that the list of $m$ above where $\sum_{r \in \mathbb{Z}} H(4n - mr^2)$ can be related to $\sigma_1(n,\chi_m)$ is complete. The vector-valued setting is useful here because it lowers the Sturm bound considerably. \\

Some examples of this occur when $m = 24,28,32,40$. For $m = 24$ it is not true that $E_2 = Q_{2,1/m,\beta}$; however, the $\mathfrak{e}_0$-components of these series have the same coefficients of $q^n$ when $n \equiv 5,7\, (\bmod \, 8)$. This can be proved by writing $f = E_2 - Q_{2,1/m,\beta}$ and considering the form $\sum_{k \in \mathbb{Z}/(24 \cdot 8)\mathbb{Z}} f(\tau + k/8) \mathbf{e}(3k/8)$, which is a modular form for (a representation of) $\Gamma_1(64)$ all of whose coefficients vanish except for those of $q^n \mathfrak{e}_0$ with $n \equiv 5$ mod $8$. To check that it vanishes identically we consider coefficients $n \equiv 5 \, (8)$ up to the Sturm bound $\frac{2}{12} [SL_2(\mathbb{Z}) : \Gamma_1(64)] = 512;$ this was done in SAGE. The case $n \equiv 7$ mod $8$ is similar. Specializing to the $n$ with $\left( \frac{24}{n} \right) = -1$, we obtain $$\sum_{r \in \mathbb{Z}} H(4n - 6r^2) = \sum_{r \in \mathbb{Z}} H(4n - 24r^2) = \frac{1}{2} \sigma_1(n,\chi_{24})$$ for all $n \equiv 7,13\, (\text{mod} \, 24)$.\\ 

For $m = 28$ the series have the same coefficients when $n \equiv 3,5,6 \, (\text{mod} \, 7)$, and specializing to the $n$ with $\left( \frac{28}{n} \right) = -1$ gives $$\sum_{r \in \mathbb{Z}} H(4n - 7r^2) = \sum_{r \in \mathbb{Z}} H(4n - 28r^2) = \frac{1}{2} \sigma_1(n,\chi_{28})$$ for all $n \equiv 5,13,17 \, (\text{mod} \, 28).$ \\ For $m=32$ it is not true that $E_2 = Q_{2,1/m,\beta};$ however, the $\mathfrak{e}_0$-components are the same, and we find $$\sum_{r \in \mathbb{Z}} H(4n - 32r^2) = \frac{1}{2} \sigma_1(n,\chi_{32}), \; n \equiv 1 \, (\text{mod}\, 4), \; \; \sum_{r \in \mathbb{Z}} H(4n - 32r^2) = \frac{2}{3} \sigma_1(n,\chi_{32}), \; n \equiv 3 \, (\text{mod}\, 8)$$ by considering odd $n$ for which $32r^2 - 4n = a^2$ is unsolvable in integers $(a,r)$. \\ For $m = 40$ the series have the same coefficients when $n \equiv 3,5 \, (\bmod\, 8)$, and specializing to the $n$ with $\left( \frac{40}{n} \right) = -1$ gives $$\sum_{r \in \mathbb{Z}} H(4n - 10r^2) = \sum_{r \in \mathbb{Z}} H(4n - 40r^2) = \frac{1}{2} \sigma_1(n,\chi_{40})$$ for all $n \equiv 11,19,21,29 \, (\text{mod} \, 40).$ \\

However, there do not seem to be any such relations of this type for $m = 44$ (or indeed for ``most" large enough $m$); in particular, $\sum_{r \in \mathbb{Z}} H(4n -11r^2)$ is not obviously related to a twisted divisor sum within any congruence class mod $44$. 
\end{rem}

\section{Restricted sums of class numbers}

Restrictions of the sums that occur in the Kronecker-Hurwitz relation to congruence classes, i.e. sums of the form $$\sum_{r \equiv a \, (d)} H(4n - r^2),$$ have been evaluated in \cite{BK2}, \cite{six} for $d=2,3,5,7$, where identities are obtained for all $a$ and $d=2,3,5$ and for some $a$ when $d=7$. These identities can be derived from the fact that the Hirzebruch-Zagier series equals the Eisenstein series when $m=4,9,25$ and that some coefficients agree when $m=49$. Here we need to compare coefficients of components $\mathfrak{e}_{\gamma}$ with $Q(\gamma) \in \mathbb{Z}$ but $\gamma$ not necessarily zero. \\

We illustrate this in the case $m=25$ and quadratic form $Q(x,y) = x^2 + xy - 6y^2$. The elements $\gamma \in A$ of the associated discriminant group with $Q(\gamma) \in \mathbb{Z}$ are represented by $$\gamma = (0,0), (1/5,3/5),(2/5,1/5),(3/5,4/5),(4/5,2/5) \in (\mathbb{Q}/\mathbb{Z})^2$$ and their products with the element $\beta = (-1/25,2/25) \in A$ we fix with $Q(\beta) = 1 - \frac{1}{25}$ are respectively $$\langle \gamma, \beta \rangle = 0,2/5,4/5,1/5,3/5.$$ In view of the formula \cite{BK}, the coefficient $c(n,\gamma)$ of $q^n \mathfrak{e}_{\gamma}$ in the Eisenstein series for $5 \nmid n$ is a multiple of the divisor sum $\sigma_1(n)$ depending on a local factor at $5$, and therefore on $n \, \text{mod} \, 5$ and on $\gamma$: $$c(n,\gamma) = \begin{cases} -6 \sigma_1(n): & n \equiv 1,4 \, (5), \; \gamma = 0 \; \text{or} \; n \equiv 3 \, (5), \; \gamma = \pm (2/5,1/5) \; \text{or} \; n \equiv 2 \, (5), \; \gamma = \pm (1/5,3/5); \\ -4 \sigma_1(n): & n \equiv 2,3 \, (5), \; \gamma = 0 \; \text{or} \; n \equiv 1,2 \, (5), \; \gamma = \pm  (2/5,1/5) \; \text{or} \; n \equiv 3,4 \, (5), \; \gamma = \pm (1/5,3/5); \\ -5 \sigma_1(n): & n \equiv 1,4 \, (5), \; \gamma = \pm (1/5,3/5) \; \text{or} \; n \equiv 4 \, (5), \; \gamma = \pm (2/5,1/5). \end{cases}$$

Since there are no cusp forms, $c(n,\gamma)$ equals the coefficient \begin{align*} &\quad -12 \sum_{r \in \mathbb{Z} - \langle \gamma, \beta \rangle} H(4n - 25r^2) - 30 \sum_{\substack{r \in \mathbb{Z} - \langle \gamma, \beta \rangle \\ 25r^2 - 4n = \square}} \Big( |r| - \sqrt{r^2 - 4n / 25} \Big) + \begin{cases} 12 \sqrt{n}: & \exists r \in \mathbb{Z} - \langle \gamma, \beta \rangle \, \text{with} \, 25r^2 = 4n; \\ 0: & \text{else}; \end{cases} \\ &= -12 \sum_{r \equiv 5 \langle \gamma, \beta \rangle \, \text{mod} \, 5} H(4n - r^2) - 6 \sum_{\substack{r \equiv 5 \langle \gamma, \beta \rangle \\ r^2 - 4n = \square}} \Big( |r| - \sqrt{r^2 - 4n} \Big) + \begin{cases} 12 \sqrt{n}: & \exists r \equiv 5 \langle \gamma, \beta \rangle \, \text{mod} \, 5 \; \text{with} \, r^2 = 4n; \\ 0: & \text{else}; \end{cases} \end{align*} of the Hirzebruch-Zagier series. Here $\frac{1}{2}(|r| - \sqrt{r^2 - 4n})$, $r \equiv \pm 5 \langle \gamma, \beta \rangle \, \text{mod}\, 5$ runs through the values $\min(d,n/d)$ for divisors $d$ of $n$ with $|r| = d + n/d \equiv \pm 5 \langle \gamma, \beta \rangle$, but it double-counts $\sqrt{n}$ if that occurs at all; so we can rewrite this as $$-12 \sum_{r \equiv 5 \langle \gamma, \beta \rangle \, \text{mod} \, 5} H(4n - r^2) - 12\varepsilon_{\gamma} \times \sum_{\substack{d | n \\ d + n/d \equiv \pm 5 \langle \gamma, \beta \rangle}} \min(d,n/d),$$ where $\varepsilon_{\gamma} = 1$ if $\gamma = 0$ and $\varepsilon_{\gamma} = \frac{1}{2}$ otherwise.

This yields the following formulas:

\begin{prop} For any $5 \nmid n$ and $a \in \mathbb{Z}/5\mathbb{Z}$, \begin{align*} &\quad \sum_{r \equiv a \, (5)} H(4n - r^2) + \varepsilon_a \sum_{\substack{d | n \\ d + n/d \equiv \pm a \, (5)}} \min(d,n/d)  \\ &= \begin{cases} \frac{1}{2} \sigma_1(n): & (n,a) \equiv (\pm 1,0), (3,\pm 1), (2,\pm 2) \, \mathrm{mod} \, 5; \\ \frac{1}{3} \sigma_1(n): & (n,a) \equiv (\pm 2,0), (1,\pm 1), (2,\pm 1), (3,\pm 2), (4,\pm 2) \, \mathrm{mod} \, 5; \\ \frac{5}{12} \sigma_1(n): & (n,a) \equiv (4,\pm 1), (1, \pm 2) \, \mathrm{mod} \, 5; \end{cases} \end{align*} where $\varepsilon_a = 1$ if $a = 0$ and $\varepsilon_a = 1/2$ otherwise.
\end{prop}

\bibliographystyle{plainnat}
\bibliography{\jobname}

\begin{thebibliography}{17}
\providecommand{\natexlab}[1]{#1}
\providecommand{\url}[1]{\texttt{#1}}
\expandafter\ifx\csname urlstyle\endcsname\relax
  \providecommand{\doi}[1]{doi: #1}\else
  \providecommand{\doi}{doi: \begingroup \urlstyle{rm}\Url}\fi

\bibitem[Apostol(1976)]{A}
Tom Apostol.
\newblock \emph{Introduction to analytic number theory}.
\newblock Springer-Verlag, New York-Heidelberg, 1976.
\newblock Undergraduate Texts in Mathematics.

\bibitem[Bringmann and Kane(2013)]{BK2}
Kathrin Bringmann and Ben Kane.
\newblock Sums of class numbers and mixed mock modular forms.
\newblock Preprint, 2013.
\newblock URL \url{arxiv:1305.0112}.

\bibitem[Brown et~al.(2008)Brown, Calkin, Flowers, James, Smith, and
  Stout]{six}
Brittany Brown, Neil Calkin, Timothy Flowers, Kevin James, Ethan Smith, and Amy
  Stout.
\newblock Elliptic curves, modular forms, and sums of {H}urwitz class numbers.
\newblock \emph{J. Number Theory}, 128\penalty0 (6):\penalty0 1847--1863, 2008.
\newblock ISSN 0022-314X.
\newblock \doi{10.1016/j.jnt.2007.10.008}.
\newblock URL \url{http://dx.doi.org/10.1016/j.jnt.2007.10.008}.

\bibitem[Bruinier(2002{\natexlab{a}})]{B}
Jan~Hendrik Bruinier.
\newblock \emph{Borcherds products on {O}(2, {$l$}) and {C}hern classes of
  {H}eegner divisors}, volume 1780 of \emph{Lecture Notes in Mathematics}.
\newblock Springer-Verlag, Berlin, 2002{\natexlab{a}}.
\newblock ISBN 3-540-43320-1.
\newblock \doi{10.1007/b83278}.
\newblock URL \url{http://dx.doi.org/10.1007/b83278}.

\bibitem[Bruinier(2002{\natexlab{b}})]{B2}
Jan~Hendrik Bruinier.
\newblock On the rank of {P}icard groups of modular varieties attached to
  orthogonal groups.
\newblock \emph{Compositio Math.}, 133\penalty0 (1):\penalty0 49--63,
  2002{\natexlab{b}}.
\newblock ISSN 0010-437X.
\newblock \doi{10.1023/A:1016357029843}.
\newblock URL \url{http://dx.doi.org/10.1023/A:1016357029843}.

\bibitem[Bruinier and Bundschuh(2003)]{BB}
Jan~Hendrik Bruinier and Michael Bundschuh.
\newblock On {B}orcherds products associated with lattices of prime
  discriminant.
\newblock \emph{Ramanujan J.}, 7\penalty0 (1-3):\penalty0 49--61, 2003.
\newblock ISSN 1382-4090.
\newblock \doi{10.1023/A:1026222507219}.
\newblock URL \url{http://dx.doi.org/10.1023/A:1026222507219}.
\newblock Rankin memorial issues.

\bibitem[Bruinier and Kuss(2001)]{BK}
Jan~Hendrik Bruinier and Michael Kuss.
\newblock Eisenstein series attached to lattices and modular forms on
  orthogonal groups.
\newblock \emph{Manuscripta Math.}, 106\penalty0 (4):\penalty0 443--459, 2001.
\newblock ISSN 0025-2611.
\newblock \doi{10.1007/s229-001-8027-1}.
\newblock URL \url{http://dx.doi.org/10.1007/s229-001-8027-1}.

\bibitem[Bruinier et~al.(2016)Bruinier, Ehlen, and Freitag]{BEF}
Jan~Hendrik Bruinier, Stephan Ehlen, and Eberhard Freitag.
\newblock Lattices with many {B}orcherds products.
\newblock \emph{Math. Comp.}, 85\penalty0 (300):\penalty0 1953--1981, 2016.
\newblock ISSN 0025-5718.
\newblock \doi{10.1090/mcom/3059}.
\newblock URL \url{http://dx.doi.org/10.1090/mcom/3059}.

\bibitem[Cohen(1993)]{C}
Henri Cohen.
\newblock \emph{A course in computational algebraic number theory}, volume 138
  of \emph{Graduate Texts in Mathematics}.
\newblock Springer-Verlag, Berlin, 1993.
\newblock ISBN 3-540-55640-0.
\newblock \doi{10.1007/978-3-662-02945-9}.
\newblock URL \url{http://dx.doi.org/10.1007/978-3-662-02945-9}.

\bibitem[Hirzebruch and Zagier(1976)]{HZ}
Friedrich Hirzebruch and Don Zagier.
\newblock Intersection numbers of curves on {H}ilbert modular surfaces and
  modular forms of {N}ebentypus.
\newblock \emph{Invent. Math.}, 36:\penalty0 57--113, 1976.
\newblock ISSN 0020-9910.
\newblock \doi{10.1007/BF01390005}.
\newblock URL \url{http://dx.doi.org/10.1007/BF01390005}.

\bibitem[Imamo\u{g}lu et~al.(2014)Imamo\u{g}lu, Raum, and Richter]{ORR}
\"Ozlem Imamo\u{g}lu, Martin Raum, and Olav Richter.
\newblock Holomorphic projections and {R}amanujan's mock theta functions.
\newblock \emph{Proc. Natl. Acad. Sci. USA}, 111\penalty0 (11):\penalty0
  3961--3967, 2014.
\newblock ISSN 1091-6490.
\newblock URL \url{https://doi.org/10.1073/pnas.1311621111}.

\bibitem[Mertens(2014)]{M2}
Michael Mertens.
\newblock Mock modular forms and class number relations.
\newblock \emph{Res. Math. Sci.}, 1:\penalty0 Art. 6, 16, 2014.
\newblock ISSN 2197-9847.
\newblock \doi{10.1186/2197-9847-1-6}.
\newblock URL \url{http://dx.doi.org/10.1186/2197-9847-1-6}.

\bibitem[Mertens(2016)]{M1}
Michael Mertens.
\newblock Eichler-{S}elberg type identities for mixed mock modular forms.
\newblock \emph{Adv. Math.}, 301:\penalty0 359--382, 2016.
\newblock ISSN 0001-8708.
\newblock \doi{10.1016/j.aim.2016.06.016}.
\newblock URL \url{http://dx.doi.org/10.1016/j.aim.2016.06.016}.

\bibitem[Scheithauer(2009)]{Sch}
Nils Scheithauer.
\newblock The {W}eil representation of {${\rm SL}_2(\Bbb Z)$} and some
  applications.
\newblock \emph{Int. Math. Res. Not. IMRN}, \penalty0 (8):\penalty0 1488--1545,
  2009.
\newblock ISSN 1073-7928.

\bibitem[Williams(2017{\natexlab{a}})]{W1}
Brandon Williams.
\newblock Poincar\'e square series for the {W}eil representation.
\newblock Preprint, 2017{\natexlab{a}}.
\newblock URL \url{https://arxiv.org/abs/1704.06758}.

\bibitem[Williams(2017{\natexlab{b}})]{W2}
Brandon Williams.
\newblock Poincar\'e square series of small weight.
\newblock Preprint, 2017{\natexlab{b}}.
\newblock URL \url{arxiv:1707.06582}.

\bibitem[Williams(2017{\natexlab{c}})]{W3}
Brandon Williams.
\newblock Vector-valued {E}isenstein series of small weight.
\newblock Preprint, 2017{\natexlab{c}}.
\newblock URL \url{arxiv:1706.03738}.

\end{thebibliography}
\end{document}